\documentclass[preprint,authoryear]{elsarticle}
\usepackage[T1]{fontenc}
\usepackage[utf8]{inputenc}
\usepackage{tabularx}
\usepackage{supertabular}
\usepackage{booktabs}
\RequirePackage[authoryear]{natbib}
\usepackage{lmodern}
\usepackage{setspace}
\usepackage{nicefrac}
\usepackage{amsmath,dsfont}
\usepackage{mathtools,mathrsfs}
\usepackage{a4wide}
\usepackage{amssymb}
\usepackage{amsthm}
\usepackage{framed}
\usepackage{bbm}
\usepackage{changes}
\usepackage[flushleft]{threeparttable}
\definecolor{dblue}{rgb}{0.21,0.21,0.55}
\usepackage{multirow}
\usepackage{lscape}
\usepackage[english]{babel}
\usepackage{graphicx}
\usepackage{arydshln}
\usepackage{rotating}
\usepackage[flushmargin,hang]{footmisc}
\usepackage{amsmath}
\usepackage{lscape}
\usepackage{marginnote,letltxmacro}
\usepackage[pdflinkmargin=5pt,pdfstartview={FitBH -32768},plainpages=false]{hyperref}
\usepackage{xcolor}
\usepackage{bbm}

\renewcommand{\P}{\mathbb{P}}

\newcommand{\R}{\mathbb{R}}

\DeclareMathAccent{\verywidehat}{\mathord}{largesymbols}{'144}

\allowdisplaybreaks[3]

\newtheorem{prop}{Proposition}[section]
\newtheorem{lem}{Lemma}

\newtheorem{rem}[prop]{Remark}

\begin{document}
\renewcommand*{\thefootnote}{\fnsymbol{footnote}}

\title{Gumbel convergence of the maximum of convoluted half-normally distributed random variables}
\author[1]{Markus Bibinger\footnote{Financial support from the Deutsche Forschungsgemeinschaft (DFG) under grant 403176476 is gratefully acknowledged.}}
\address[1]{Faculty of Mathematics and Computer Science, Julius-Maximilians-Universit\"at W\"urzburg} 
\normalsize
\begin{frontmatter}

\vspace{-.7cm} 

\begin{abstract}
{{\normalsize \noindent
In this note, we establish the convergence in distribution of the maxima of i.i.d.\ random variables to the Gumbel distribution with the associated normalizing sequences for several examples that are related to the normal distribution. Motivated by tests for jumps in high-frequency data, our main interest is in the half-normal distribution and the sum or difference of two independent half-normally distributed random variables. Since the half-normal distribution is neither stable nor symmetric, these examples are non-obvious generalizations. It is shown that the sum and difference of two independent half-normally distributed random variables and other examples yield distributions with tail behaviours that relate to the normal case. It turns out that the Gumbel convergence for all such distributions can be proved following similar steps. We illustrate the results in Monte Carlo simulations.
}}
\begin{keyword}
Convolution tails \sep extreme value theory\sep  Gumbel convergence \sep half-normal distribution \\[.25cm] 
\MSC[2010] 60G70
\end{keyword}
\end{abstract}

\end{frontmatter}
\thispagestyle{plain}

\onehalfspacing
\section{Introduction\label{sec:1}}
From extreme value theory it is known that non-degenerate limit distributions of normalized maxima of i.i.d.\  real-valued, absolutely continuously distributed random variables are always extreme value distributions of Weibull, Fr\'{e}chet or Gumbel type, see \cite{Gnedenko}. If $F$ denotes the cumulative distribution function (cdf) of such i.i.d.\ random variables, we write $F\in \text{MDA}(G_{\gamma})$, that is, $F$ is in the \emph{maximum domain of attraction} of the generalized extreme value distribution which is specified by the extreme value index $\gamma\in\R$. If the right-end point of $F$ is $+\infty$, the limit is a Fr\'{e}chet distribution, $\gamma>0$, if $F$ has Pareto tails and Gumbel, $\gamma=0$, if $F$ has exponential tails. This note considers cdfs $F$ with exponential tails, such that $F\in \text{MDA}(\Lambda)$, where we use the standard notation $\Lambda$ for the Gumbel distribution $G_{0}$. We write $Z\stackrel{d}{\sim} F$, if the random variable $Z$ has cdf $F$, and $Z\sim\mathcal{N}(\mu,\sigma^2)$, if $Z$ is normally distributed with mean $\mu\in\R$, and variance $\sigma^2>0$. $Z_i\stackrel{i.i.d.}{\sim} F$, and $Z_i\stackrel{i.i.d.}{\sim}\mathcal{N}(\mu,\sigma^2)$, means that all random variables $Z_i$ of a sequence $(Z_i)_{1\le i\le n}$ are independent and identically distributed. We write
\begin{align}\label{gumbel}
\frac{\max_{1\le i\le n}Z_i-b_n}{a_n}\stackrel{d}{\longrightarrow} \Lambda\,,
\end{align}
for the convergence in distribution to the Gumbel limit distribution, i.e.\ when it holds for all $x\in\R$ that 
\begin{align}
\lim_{n\to\infty}\P\Big(a_n^{-1}\Big(\max_{1\le i\le n}Z_i-b_n\Big)\Big)\le x\Big)= \exp\big({-e^{-x}}\big)\,.
\end{align}
For $Z\sim\mathcal{N}(0,\sigma^2)$, $|Z|$ has a \emph{half-normal distribution}. The density of the half-normal distribution
\begin{align}\label{hnden}f_{H\negthinspace N}(x)=\sqrt{\frac{2}{\pi}}e^{-x^2/2}~,~x\ge 0,\end{align}
is easily obtained from the density of the normal distribution exploiting its symmetry. The distribution of $Z=X+Y$ is the convolution of the distributions of $X$ and $Y$. The convolution of two normally distributed random variables $X$ and $Y$ is again normally distributed. Since the half-normal distribution does not satisfy such a stability, the extreme value convergences for $|X|+|Y|$ and $|X|-|Y|$ are not obvious and established in this note together with some related examples. For all examples, we establish a sufficient condition that the distributions are in the maximum domain of attraction of the Gumbel distribution although this is clear by the exponential tail behaviour. The main open problem solved here is rather to determine the normalizing sequences $(a_n)$ and $(b_n)$ in \eqref{gumbel} for the different examples. The analysis provided in this note leads me to the conjecture, that for all distributions with a certain tail behaviour of Gaussian type, that is, the right tails are asymptotically proportional to $h(t)\,\exp(-c t^2)$, with some rational function $h$, the Gumbel convergence can be proved by a modification of the proof given here following the same steps. Moreover, the distributions of sums and differences remain within this class.

The question about the precise result of the Gumbel convergence \eqref{gumbel}, when $Z_i$ are (absolute) differences of two independent half-normally distributed random variables, arose when working on a test for jumps in high-frequency financial order-book data which are modelled by a sum of a semi-martingale efficient price process and one-sided microstructure noise. This model has been proposed by \cite{BJR}. There, the half-normal distribution occurs as the distribution of the maximum of Brownian motion over some fix interval which by the reflection principle for Brownian motion equals the distribution of the absolute value of the end-point. Comparing values on neighboured intervals to test for jumps leads to the (absolute) difference of two independent half-normally distributed random variables. A global test for jumps is then based on the maximal (absolute) difference for which we need the result provided in Proposition \ref{nvthm}  of this note.

There are more standard methods in econometrics based on the considered examples of Gumbel convergence. The Gumbel test for price jumps in high-frequency financial data by \cite{leemykland} is very popular. The asymptotic distribution under the null hypothesis of no price jump, when the price is modelled by discrete recordings of a continuous semi-martingale, is by Lemma 1 of \cite{leemykland} the Gumbel distribution. 
However, there is a small but relevant typo in the result, since the test statistic uses \emph{absolute} returns combined with the normalizing sequences for the normal instead of the half-normal distribution. We shall see in the Monte Carlo simulations in Section \ref{sec:4} that, although the results in \eqref{g1}-\eqref{g6} look very similar at a quick glance, the differences of the normalizing sequences are important.

\section{Results and discussion\label{sec:2}}

\begin{prop}\label{nvthm}
Let $(X_1,\ldots ,X_n,Y_1,\ldots, Y_n)$ be a $2n$-dimensional vector of i.i.d.\ standard normally distributed random variables. 
\begin{enumerate}
\begin{subequations}
\item If $Z_i=X_i$, \eqref{gumbel} holds with the sequences
\begin{align}\label{g1}
a_n=\frac{1}{\sqrt{2\log (n)}}~, ~\mbox{and}~~b_n=\sqrt{2\log( n)}+\delta_n~, ~\mbox{with}~~\delta_n=-\frac{\log(4\pi\log (n))}{2\sqrt{2\log (n)}}\,.
\end{align}
\item If $Z_i=|X_i|$, \eqref{gumbel} holds with the sequences
\begin{align}\label{g2}
a_n=\frac{1}{\sqrt{2\log (2n)}}~, ~\mbox{and}~~b_n=\sqrt{2\log (2n)}+\delta_n~, ~\mbox{with}~~\delta_n=-\frac{\log(4\pi\log (2n))}{2\sqrt{2\log (2n)}}\,.
\end{align}
\item If $Z_i=X_i+Y_i$, and as well for $Z_i=X_i-Y_i$, \eqref{gumbel} holds with the sequences
\begin{align}\label{g3}
a_n=\frac{1}{\sqrt{\log (n)}}~, ~\mbox{and}~~b_n=2\sqrt{\log( n)}+\delta_n~, ~\mbox{with}~~\delta_n=-\frac{\log(4\pi\log (n))}{2\sqrt{\log (n)}}\,.
\end{align}
\item If $Z_i=|X_i|+|Y_i|$, \eqref{gumbel} holds with the sequences
\begin{align}\label{g4}
a_n=\frac{1}{\sqrt{\log (4n)}}~, ~\mbox{and}~~b_n=2\sqrt{\log (4n)}+\delta_n~, ~\mbox{with}~~\delta_n=-\frac{\log(4\pi\log (4n))}{2\sqrt{\log (4n)}}\,.
\end{align}
\item If $Z_i=|X_i|-|Y_i|$, \eqref{gumbel} holds with the sequences
\begin{align}\label{g5}
a_n=\frac{1}{\sqrt{2\log (n)}}~, ~\mbox{and}~~b_n=\sqrt{2\log (n)}+\delta_n~, ~\mbox{with}~~\delta_n=-\frac{\log(\pi\log (n))}{\sqrt{2\log (n)}}\,.
\end{align}
\item If $Z_i=\big||X_i|-|Y_i|\big|$, \eqref{gumbel} holds with the sequences
\begin{align}\label{g6}
a_n=\frac{1}{\sqrt{2\log (2n)}}~, ~\mbox{and}~~b_n=\sqrt{2\log (2n)}+\delta_n~, ~\mbox{with}~~\delta_n=-\frac{\log(\pi\log (2n))}{\sqrt{2\log (2n)}}\,.
\end{align}
\end{subequations}
\begin{rem}
While the normal distribution in \eqref{g1} is a well-known example for a distribution in the maximum domain of attraction of the Gumbel distribution, the half-normal case \eqref{g2} and the (absolute) difference and the sum of independent half-normally distributed random variables, \eqref{g4}-\eqref{g6}, have so far not been focussed on in the literature. The half-normal case \eqref{g2} might appear intuitive by the symmetry of the normal distribution. The convergence of the maximum of i.i.d.\ random variables which are defined as sums is usually not readily obtained from known extreme value convergences of the summands. For the normal distribution, however, \eqref{g3} is obvious by \eqref{g1} with the convolution property that $X_i+Y_i\stackrel{d}{\sim} \mathcal{N}(0,2)$. Let me point out that $|X_i|+|Y_i|$ and $|X_i|-|Y_i|$ instead are not half-normally distributed, that is, the half-normal distribution is not a stable distribution.  Nevertheless, the convergences \eqref{g4}-\eqref{g6} can be proved using similar ingredients as for \eqref{g1} and an asymptotic expansion of the convolution integrals. Naturally, in contrast to \eqref{g3}, which applies to the sum and difference, \eqref{g4} and \eqref{g5} are different, since the half-normal distribution is not symmetric. 
\end{rem}
\end{enumerate}
\end{prop}
\section{Proof of the results\label{sec:3}}
We begin with a lemma which is used to exploit the symmetry of a distribution in a simple way.
\begin{lem}\label{lem1}
Suppose that $Z_i\stackrel{i.i.d.}{\sim} F\in \text{MDA}(G_{\gamma})$, with
\begin{align}\label{leman}
\frac{\max_{1\le i\le n}Z_i-b_n}{a_n}\stackrel{d}{\longrightarrow} G_{\gamma}\,,
\end{align}
with the (cdf of the) generalized extreme value distribution $G_{\gamma}$, and let $F$ be symmetric such that $Z_i$ and $-Z_i$ are identically distributed, $Z_i\stackrel{d}{=}-Z_i$. For $|Z_i|\stackrel{i.i.d.}{\sim} F_{|\cdot|}$, it follows that $ F_{|\cdot|}\in \text{MDA}(G_{\gamma})$, with
\begin{align}\label{lemaus}
\frac{\max_{1\le i\le n}|Z_i|-b_{2n}}{a_{2n}}\stackrel{d}{\longrightarrow} G_{\gamma}\,.
\end{align}
\end{lem}
\begin{proof}
For a cdf $F$, we write $\bar F(x)=1-F(x)$ for the tail function. By the symmetry, it holds for $x\ge 0$ that
\[\P(|Z_i|>x)=\P(Z_i>x)+\P(-Z_i>x)=2\,\P(Z_i>x)\,,\]
such that $ \bar F_{|\cdot|}(x)=2\bar F(x),x\ge 0$. Equivalent to \eqref{leman} is the convergence
\[\lim_{n\to\infty} n\bar F(a_n t+b_n)=-\log\big(G_{\gamma}(t) \big)\,,\]
see Eq.\ (1.1.3) in \cite{haan}. The convergence
\[\lim_{n\to\infty} n\bar F_{|\cdot|}(\tilde a_n t+\tilde b_n)=\lim_{n\to\infty} 2n\bar F(\tilde a_n t+\tilde b_n)=-\log\big(G_{\gamma}(t) \big)\,,\]
is thus satisfied setting $\tilde a_n=a_{2n}$, and $\tilde b_n=b_{2n}$. We conclude \eqref{lemaus}.
\end{proof}
For asymptotic equivalence of two positive functions $f$ and $g$, we write $f\asymp g$, which means that
\[\lim_{x\to\infty}\frac{f(x)}{g(x)}=1\,.\]
The second lemma provides a neat asymptotic relation between tail functions and densities of normal and half-normal distributions.
\begin{lem}\label{lem2}For some constant $c>0$, it holds true that
\[\int_x^{\infty} e^{-ct^2}\,\mathrm{d} t\asymp \frac{1}{2c}\frac{1}{x} e^{-cx^2}\,.\]
\end{lem}
\begin{proof}
The chain or product rule of differentiation yield that
\[\frac{\mathrm{d} }{\mathrm{d} x}\frac{e^{-cx^2}}{x}=-2c\,e^{-cx^2}-x^{-2}e^{-cx^2}\,.\]
L'Hospital's rule hence yields that
\begin{align*}
\lim_{x\to\infty}\frac{\int_x^{\infty} e^{-ct^2}\,\mathrm{d} t}{x^{-1} e^{-cx^2}}&=\lim_{x\to\infty}\frac{-e^{-cx^2}}{-x^{-2}e^{-cx^2}-2c\,e^{-cx^2}}\\
&=\lim_{x\to\infty}\frac{1}{x^{-2}+2c}=\frac{1}{2c}\,,
\end{align*}
what proves the claim.
\end{proof}
We first prove \eqref{g1}. There are different proofs for this result available in the literature, see, for instance, Example 1.1.7 in \cite{haan}. We follow the proof from Section 5.4 of \cite{kabluchko}, which can be generalized with Lemmas \ref{lem1} and \ref{lem2} below to prove \eqref{g2}-\eqref{g6}. Denote by
\[\varphi(x)=\frac{1}{\sqrt{2\pi}}\,e^{-x^2/2}~,~\Phi(x)=\int_{-\infty}^x\varphi(y)\,\mathrm{d} y\,,\]
the density and the cdf of the standard normal distribution, respectively. Lemma \ref{lem2} yields that $1-\Phi(x)\asymp \varphi(x)/x$. By Theorem 1.2.1, Equation (1.2.4), in \cite{haan}, \eqref{gumbel} is satisfied if and only if there exists a function $f$, such that for all $x\in\R$, the cdf $F$ of the random variables $(Z_i)$ satisfies
\begin{align}\,\lim_{t\uparrow x^*}\frac{1-F(t+xf(t))}{1-F(t)}=e^{-x}~.\label{crit}\end{align}
$x^*$ is the right end-point of the distribution which is $x^*=+\infty$ in our cases. We verify this condition for $F=\Phi$ with $f(t)=t^{-1}$:
\[\lim_{t\uparrow\infty}\frac{1-\Phi\Big(t+\tfrac{x}{t}\Big)}{1-\Phi(t)}=\lim_{t\uparrow\infty}\frac{\Big(t+\tfrac{x}{t}\Big)^{-1}\exp{\big(-t^2/2-x-x^2/(2t^2)\big)}}{t^{-1}e^{-t^2/2}}=e^{-x}\,,\,\forall\,x\in\R\,.\]
We conclude that $\Phi\in\text{MDA}(\Lambda)$. We determine the sequences $(a_n)$ and $(b_n)$ in \eqref{g1}. We can start with
\begin{align}\label{nv1}\lim_{n\to\infty}n\big(1-\Phi(a_n t+ b_n)\big)=-\log(\Lambda(t))=e^{-t}\,,\end{align}
or directly use $b_n=U(n)$, with $U$ the general notation for the left-continuous generalized inverse of $1/(1-F)$, which is known from extreme value theory, see Remark 1.1.9 in \cite{haan}. This yields that
\begin{align}n=\frac{1}{1-\Phi(b_n)}\asymp\sqrt{2\pi}b_ne^{b_n^2/2}\,.\end{align}
We set $b_n=\sqrt{2\log(n)}+\delta_n$, with a null sequence $\delta_n$. Inserting this gives
\[n\asymp\sqrt{2\pi}\sqrt{2\log(n)}\exp\big(\log(n)+\sqrt{2\log(n)}\delta_n\big)\,,\]
and we find that the identity holds true for
\[\delta_n=\frac{-\log(4\pi\log(n))}{2\sqrt{2\log(n)}}\,.\]
The sequence $(a_n)$ can be obtained in general by $a_n=n\,U^{\prime}(n)$, see Remark 1.1.9 in \cite{haan}. Here, we derive that 
\[a_n=(2\log(n))^{-1/2}~,\] 
what can be deduced directly from \eqref{nv1}.\\[.2cm]
Next, \eqref{g2} is readily implied by \eqref{g1} and Lemma \ref{lem1}.\\[.2cm]
Since
\begin{align*}
\frac{\max_{1\le i\le n}Z_i-b_n}{a_n}\stackrel{d}{\longrightarrow} G_{\gamma}
\end{align*}
is for any constant $\sigma>0$ equivalent to
\begin{align*}
\frac{\max_{1\le i\le n}\sigma Z_i-\sigma b_n}{\sigma a_n}\stackrel{d}{\longrightarrow} G_{\gamma}\,,
\end{align*}
with $\sigma Z_i\stackrel{i.i.d.}{\sim}\mathcal{N}(0,\sigma^2)$, the convolution property of the normal distribution, that $X_i+Y_i\stackrel{i.i.d.}{\sim}\mathcal{N}(0,2)$, gives \eqref{g3}. By symmetry $Y_i\stackrel{d}{=} -Y_i$, such that the sum and the difference have the same distribution.\\[.2cm]
If $f_{H\negthinspace N}$ denotes the density of the half-normal distribution from \eqref{hnden}, we write
\[f_{H\negthinspace N}^{*2}=f_{H\negthinspace N}*f_{H\negthinspace N}\,,\]
for the convolution square determined by the equation
\[f_{H\negthinspace N}^{*2}(x)=\int_0^{\infty}f_{H\negthinspace N}(u)\,f_{H\negthinspace N}(x-u)\,\mathrm{d}u\,.\]
The integral has an explicit solution which we express using the $\text{erf}$ and $\text{erfc}$ functions associated with $\Phi$ by the identities
\[\Phi(x)=\frac{1+\text{erf}(x/\sqrt{2})}{2}~,~\text{erfc}(x)=1-\text{erf}(x)~.\]
We obtain that
\begin{align*}f_{H\negthinspace N}^{*2}(x)&=\frac{2}{\pi}\int_0^{\infty}e^{-u^2/2}\,e^{-(x-u)^2/2}\,\mathrm{d}u\,\\
&=\frac{2}{\pi}\,e^{-x^2/2}\int_0^{\infty}e^{-u^2}\,e^{xu}\,\mathrm{d}u\,\\
&=\frac{2}{\pi}\,e^{-x^2/4}\int_0^{\infty}e^{-(u-x/2)^2}\,\mathrm{d}u\,\\
&=\frac{\sqrt{2}}{\pi}\,e^{-x^2/4}\int_{-x/\sqrt{2}}^{\infty}e^{-v^2/2}\,\mathrm{d}v\,\\
&=\frac{2}{\sqrt{\pi}}\,e^{-x^2/4}\Big(1-\Phi\big(-x/\sqrt{2}\big)\Big)\,\\
&=\frac{2}{\sqrt{\pi}}\,e^{-x^2/4}\,\Phi\big(x/\sqrt{2}\big)\,\\
&=\sqrt{\frac{1}{\pi}}\,e^{-x^2/4}\big(1+\text{erf}(x/2)\big)\\
&\asymp \frac{2}{\sqrt{\pi}}\,e^{-x^2/4}\,.
\end{align*}
This expansion is included in the more general illustration of the tails of the convolution of two densities with Gaussian tail behaviour given in \cite{gausstails} with a different proof. For the tail function $\bar F_{H\negthinspace N}^{*2}$, Lemma \ref{lem2} yields that
\[\bar F_{H\negthinspace N}^{*2}\asymp \int_x^{\infty} \frac{2}{\sqrt{\pi}}\,e^{-t^2/4}\,\mathrm{d}t\asymp \frac{4}{\sqrt{\pi}}\,\frac{e^{-x^2/4}}{x}\,.\]
Hence, \eqref{crit} is satisfied with $f(t)=2t^{-1}$, since
\[\lim_{t\uparrow\infty}\frac{\bar F_{H\negthinspace N}^{*2}\Big(t+\tfrac{2x}{t}\Big)}{\bar F_{H\negthinspace N}^{*2}(t)}=\lim_{t\uparrow\infty}\frac{\Big(t+\tfrac{2x}{t}\Big)^{-1}\exp{\big(-t^2/4-x-(x/t)^2)\big)}}{t^{-1}e^{-t^2/4}}=e^{-x}\,,\,\forall\,x\in\R\,.\]
We conclude that the convoluted half-normal distribution is in the maximum domain of attraction of the Gumbel distribution. We determine the normalizing sequences, $(a_n)$ and $(b_n)$, similarly as in the normal case, but establish some crucial differences. We know that
\begin{align*}n=\frac{1}{\bar F_{H\negthinspace N}^{*2}(b_n)}\asymp\frac{\sqrt{\pi}}{4}\,b_n\,e^{b_n^2/4}\,,\end{align*}
and set $b_n=2\sqrt{\log(4n)}+\delta_n$, with a null sequence $\delta_n$. Inserting this yields that
\[n\asymp\frac{\sqrt{\pi}}{4}\,2\sqrt{\log(4n)}\exp\big(\log(4n)+\sqrt{\log(4n)}\delta_n\big)\,,\]
and we obtain that 
\[\delta_n=\frac{-\log(4\pi\log(4n))}{2\sqrt{\log(4n)}}\,.\]
Computing $n U^{\prime}(n)$, starting with $U(n)=b_n$, gives for the sequence $(a_n)$ that 
\[a_n=(\log(4n))^{-1/2}~.\] 
We have proved \eqref{g4}.\\[.2cm]
Denote with $g$ the density of $|X_1-Y_1|$ on the positive real axis. With similar steps as for the convolution, we compute
\begin{align*}g(x)&=\frac{2}{\pi}\int_0^{\infty}e^{-u^2/2}\,e^{-(x+u)^2/2}\,\mathrm{d}u\,\\
&=\frac{\sqrt{2}}{\pi}\,e^{-x^2/4}\int_{x/\sqrt{2}}^{\infty}e^{-v^2/2}\,\mathrm{d}v\,\\
&=\frac{2}{\sqrt{\pi}}\,e^{-x^2/4}\Big(1-\Phi\big(x/\sqrt{2}\big)\Big)\,\\
&=\sqrt{\frac{1}{\pi}}\,e^{-x^2/4}\,\text{erfc}(x/2)\,.
\end{align*}
We then use that
\[ \text{erfc}(x)\asymp \frac{e^{-x^2}}{\sqrt{\pi} x}\,,\]
obtained from Lemma \ref{lem2}, to deduce that
\begin{align*}g(x)\asymp \frac{2}{\pi}\frac{e^{-x^2/2}}{x}\,.\end{align*}
Different to all previous densities, we have the additional factor $x^{-1}$ in the tail behaviour and not only an exponential decay. Using that
\begin{align*}\frac{\mathrm{d} }{\mathrm{d} x}\frac{e^{-x^2/2}}{x^2}=-\frac{e^{-x^2/2}}{x}-2x^{-3}e^{-x^2/2}\,,\end{align*}
l'Hospital's rule yields that
\[\int_x^{\infty} \frac{e^{-t^2/2}}{t}\,\mathrm{d} t\asymp \frac{e^{-x^2/2}}{x^2} \,.\]
We conclude that the associated tail function $\bar G$ satisfies
\[\bar G(x)\asymp \frac{2}{\pi}\,\frac{e^{-x^2/2}}{x^2}\,.\]
Again, \eqref{crit} is satisfied with $f(t)=t^{-1}$, since
\[\lim_{t\uparrow\infty}\frac{\bar G\Big(t+\tfrac{x}{t}\Big)}{\bar G(t)}=\lim_{t\uparrow\infty}\frac{\Big(t+\tfrac{x}{t}\Big)^{-2}\exp{\big(-t^2/2-x-x^2/(2t^2)\big)}}{t^{-2}e^{-t^2/2}}=e^{-x}\,,\,\forall\,x\in\R\,.\]
Setting $b_n=\sqrt{2\log(n)}+\delta_n$, with a null sequence $\delta_n$, we determine the normalizing sequences $(a_n)$ and $(b_n)$, based on
\[n\asymp\frac{\pi}{2}\,b_n^2\,e^{b_n^2/2}\asymp \pi\log(n)\exp\big(\log(n)+\sqrt{2\log(n)}\delta_n\big)\,.\]
We conclude that
\[\delta_n=-\frac{\log(\pi\log (n))}{\sqrt{2\log (n)}}\,,\] 
and \eqref{g5}. \eqref{g6} readily follows from \eqref{g5} with Lemma \ref{lem1}.\hfill\qed
\clearpage
\section{Monte Carlo simulations\label{sec:4}}
\begin{figure}[t]
\begin{center}
\fbox{
\includegraphics[width=6.5cm]{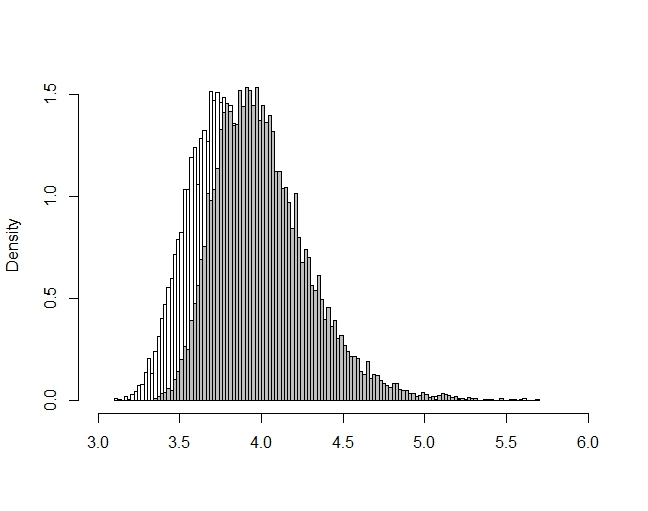}~~~~\includegraphics[width=6.5cm]{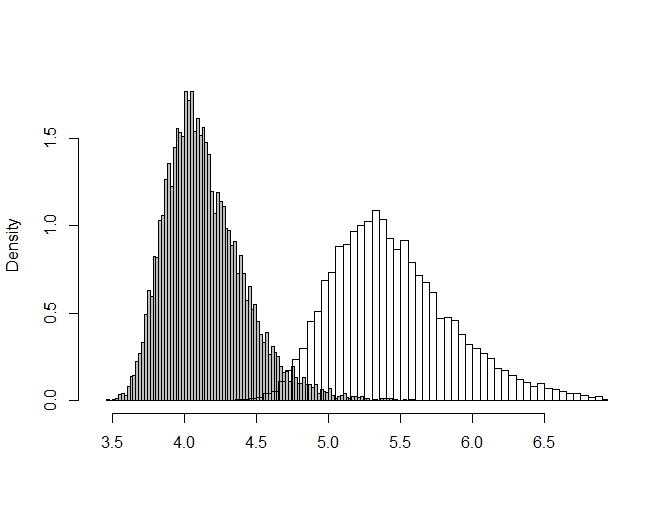}}
\caption{\label{Fig:1}Left: Histogram of $M_n^{(1)}$ (white) and $M_n^{(2)}$ (grey). 
Right: Histogram of $M_n^{(3)}$ (white) and $M_n^{(5)}$ (grey).}
\end{center}\end{figure} 
We perform Monte Carlo simulations where for each example we generate $2n=20.000$ realizations of standard normally distributed random variables $(x_1,\ldots,x_n,y_1,\ldots,y_n)$. We then take in $N=10.000$ Monte Carlo repetitions the maxima of
\begin{subequations}
\begin{align}
M_n^{(1)}&=\max_{1\le i\le n } x_i\,,\\
M_n^{(2)}&=\max_{1\le i\le n } |x_i|\,,\\
M_n^{(3)}&=\max_{1\le i\le n } \big(x_i-y_i\big)\,,\\
M_n^{(4)}&=\max_{1\le i\le n } \big(|x_i|+|y_i|\big)\,,\\
M_n^{(5)}&=\max_{1\le i\le n } \big(|x_i|-|y_i|\big)\,,\\
M_n^{(6)}&=\max_{1\le i\le n } \big||x_i|-|y_i|\big|\,.
\end{align}
\end{subequations}
In Figure \ref{Fig:1} the left plot compares histograms for the $N$ values of $M_n^{(1)}$ and $M_n^{(2)}$. The two empirical distributions are not that far from each other. This might be one reason why the normalizing sequences in \eqref{g1} and \eqref{g2} are sometimes confused without notice. However, for the precision of an asymptotic test, the different normalizing sequences are nevertheless crucial. In fact, $M_n^{(2)}$ is distributed as $M_{2n}^{(1)}$, when we have $2n$ instead of $n$ i.i.d.\ observations, and not as $M_{n}^{(1)}$. The right plot of Figure \ref{Fig:2} compares the empirical distributions of the values of $M_n^{(3)}$ and $M_n^{(5)}$. These two empirical distributions are quite different in their location and scale.

Figure \ref{Fig:2} illustrates the finite-sample precision of the asymptotic standard Gumbel limit distribution for the empirical distribution of properly normalized maxima in our six examples. The histograms are for the statistics left-hand side in Eq.\ \eqref{g1}-\eqref{g6} and the black line shows the density of the standard Gumbel limit distribution to draw a comparison. For all examples, the Gumbel limit distribution closely tracks the empirical distributions. In particular, the fit for convoluted half-normally distributed random variables appears to be as good as in the standard normal case. With their different normalizing sequences, the \emph{normalized} maxima have the same standard Gumbel limit distribution, including normalized versions of $M_n^{(3)}$ and $M_n^{(5)}$, whose empirical distributions are quite different as shown in the right plot of Figure \ref{Fig:1}.

\begin{figure}[t]
\begin{center}
\includegraphics[width=6.5cm]{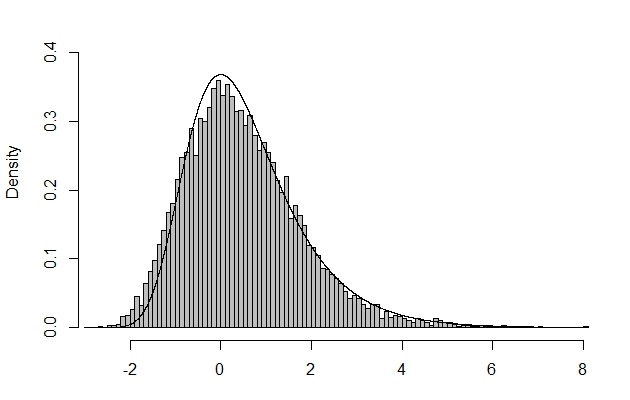}~~~~\includegraphics[width=6.5cm]{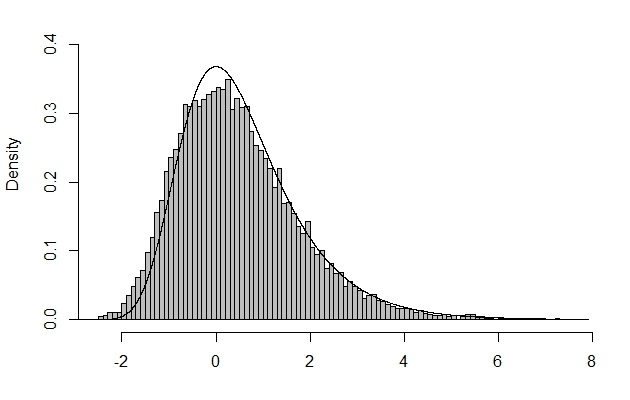}\\
\includegraphics[width=6.5cm]{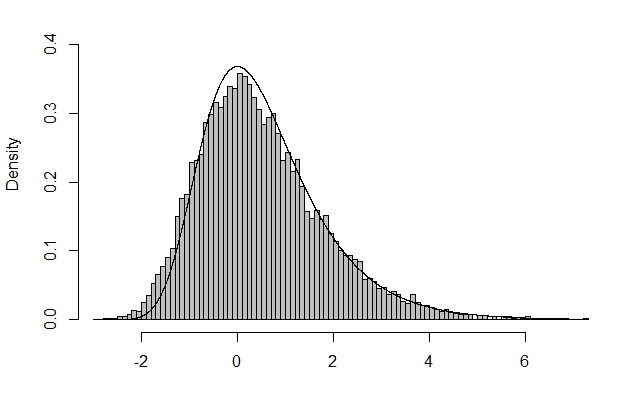}~~~~\includegraphics[width=6.5cm]{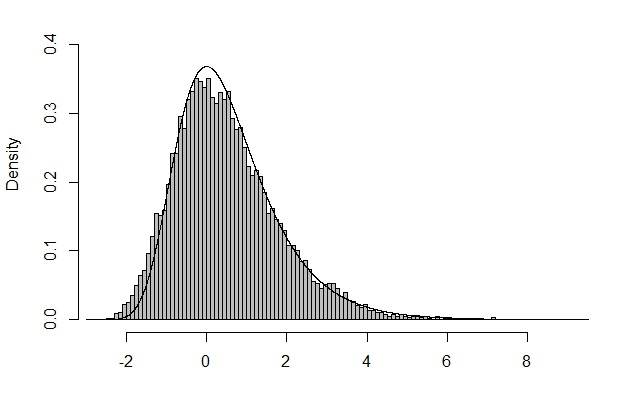}\\
\includegraphics[width=6.5cm]{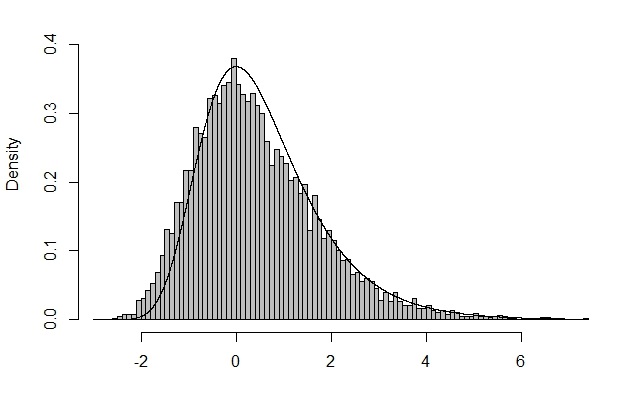}~~~~\includegraphics[width=6.5cm]{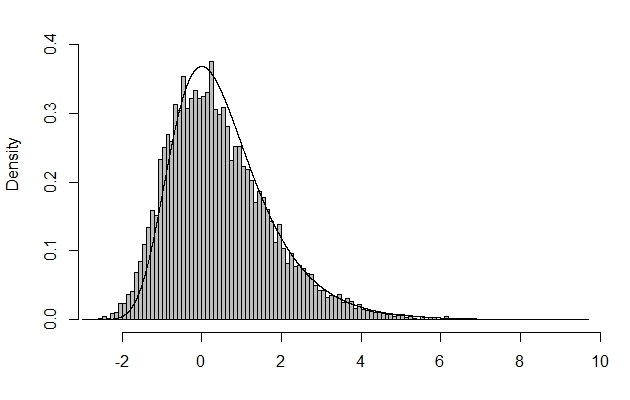}
\caption{\label{Fig:2}Histograms for the normalized maxima $M_n^{(k)}$, $k=1,\ldots,6$, from left to right and top down. The histograms give the values of the statistics left-hand side in Eq.\ \eqref{g1}-\eqref{g6}. The black line gives the density of the Gumbel limit distribution.}
\end{center}\end{figure} 

\clearpage

\bibliographystyle{chicago}
\bibliography{literatur}

\begin{thebibliography}{}

\bibitem[\protect\citeauthoryear{Balkema, Klüppelberg, and Resnick}{Balkema
  et~al.}{1993}]{gausstails}
Balkema, A.~A., C.~Klüppelberg, and S.~I. Resnick (1993).
\newblock Densities with {G}aussian tails.
\newblock {\em Proceedings of the London Mathematical Society\/}~{\em
  s3-66\/}(3), 568--588.

\bibitem[\protect\citeauthoryear{Bibinger, Jirak, and Reiß}{Bibinger
  et~al.}{2016}]{BJR}
Bibinger, M., M.~Jirak, and M.~Reiß (2016, 10).
\newblock Volatility estimation under one-sided errors with applications to
  limit order books.
\newblock {\em Ann. Appl. Probab.\/}~{\em 26\/}(5), 2754--2790.

\bibitem[\protect\citeauthoryear{{de Haan} and {Ferreira}}{{de Haan} and
  {Ferreira}}{2006}]{haan}
{de Haan}, L. and A.~{Ferreira} (2006).
\newblock {\em {Extreme value theory. An introduction.}}
\newblock New York, NY: Springer.

\bibitem[\protect\citeauthoryear{Gnedenko}{Gnedenko}{1943}]{Gnedenko}
Gnedenko, B.~W. (1943).
\newblock Sur la distribution limite du terme maximum d'une serie aleatoire.
\newblock {\em Annals of Mathematics\/}~{\em 44\/}(3), 423--453.

\bibitem[\protect\citeauthoryear{Kabluchko}{Kabluchko}{2015}]{kabluchko}
Kabluchko, Z. (2015).
\newblock {\em {Extremwerttheorie}}.
\newblock {Vorlesungsskript, Universit\"at M\"unster}.

\bibitem[\protect\citeauthoryear{Lee and Mykland}{Lee and
  Mykland}{2008}]{leemykland}
Lee, S. and P.~A. Mykland (2008).
\newblock Jumps in financial markets: A new nonparametric test and jump
  dynamics.
\newblock {\em Review of Financial Studies\/}~{\em 21}, 2535--2563.

\end{thebibliography}
\addcontentsline{toc}{section}{References}
\end{document}